\documentclass[leqno,11pt,oneside]{amsart}

\usepackage{geometry}
\usepackage{physics}
\usepackage{times}

\usepackage[all]{xy} 

\usepackage{amsmath, amssymb, amsfonts, latexsym, mdwlist, amsthm,amscd}
\usepackage{caption, subcaption}
\usepackage{tikz}
\usetikzlibrary{calc,trees,arrows,chains,shapes.geometric,%
	decorations.pathreplacing,decorations.pathmorphing,shapes,%
	matrix,shapes.symbols}

\tikzset{
	>=stealth',
	punktchain/.style={
		rectangle,
		rounded corners,
		draw=black, thick,
		minimum height=3em,
		text centered,
		on chain},
	line/.style={draw, thick, <-},
	element/.style={
		tape,
		top color=white,
		bottom color=blue!50!black!60!,
		minimum width=8em,
		draw=blue!40!black!90, very thick,
		text width=10em,
		minimum height=3.5em,
		text centered,
		on chain},
	every join/.style={->, thick,shorten >=1pt},
	decoration={brace},
	tuborg/.style={decorate},
	tubnode/.style={midway, right=2pt},
}

\usepackage{cite}
\usepackage{url}

\usepackage{verbatim}
\usepackage{mathtools}

\usepackage{pgf,tikz}
\usepackage{tikz-cd}
\usetikzlibrary{calc,matrix,patterns,shadings,backgrounds,fit}
\pgfdeclarelayer{myback}
\pgfsetlayers{myback,background,main}
\usepackage{stmaryrd}
\usepackage{diagbox}
\usepackage{extarrows}

\usepackage[bookmarks, colorlinks, breaklinks, pdftitle={contractibility}]{hyperref}
\hypersetup{linkcolor=blue,citecolor=blue,filecolor=black,urlcolor=blue}

\usepackage{paralist}
\setdefaultenum{(a)}{(i)}{}{}
\usepackage{enumitem} 

\def\C{\ensuremath{\mathbb{C}}}

\def\H{\ensuremath{\mathbb{H}}}

\def\Q{\ensuremath{\mathbb{Q}}}
\def\R{\ensuremath{\mathbb{R}}}
\def\Z{\ensuremath{\mathbb{Z}}}

\def\Alb{\mathop{\mathrm{Alb}}\nolimits}

\def\Aut{\mathop{\mathrm{Aut}}\nolimits}
\def\ch{\mathop{\mathrm{ch}}\nolimits}

\def\Coh{\mathop{\mathrm{Coh}}\nolimits}

\def\dim{\mathop{\mathrm{dim}}\nolimits}

\def\inf{\mathop{\mathrm{inf}}\nolimits}

\def\GL{\mathop{\mathrm{GL}}\nolimits}

\def\Hom{\mathop{\mathrm{Hom}}\nolimits}
\def\id{\mathop{\mathrm{id}}\nolimits}

\def\Ker{\mathop{\mathrm{Ker}}\nolimits}
\def\Kn{\mathop{\mathrm{K}_{\mathrm{num}}}}

\def\min{\mathop{\mathrm{min}}\nolimits}

\def\Pic{\mathop{\mathrm{Pic}}}

\def\rk{\mathop{\mathrm{rk}}}

\def\Stab{\mathop{\mathrm{Stab}}\nolimits}

\def\Db{\mathrm{D}^{b}}

\def\glt{\widetilde{\mathrm{GL}}^+_2(\R)}
\def\sbb{\sigma^{a,b}_{\alpha,\beta}}

\def\abs#1{\left\lvert#1\right\rvert}

\makeatletter
\newtheorem*{rep@theorem}{\rep@title}
\newcommand{\newreptheorem}[2]{%
\newenvironment{rep#1}[1]{%
 \def\rep@title{#2 \ref{##1}}%
 \begin{rep@theorem}}%
 {\end{rep@theorem}}}
\makeatother

\newtheorem{Thm}{Theorem}[section]
\newreptheorem{Thm}{Theorem}
\newtheorem{Prop}[Thm]{Proposition}

\newtheorem{Lem}[Thm]{Lemma}

\newtheorem{Cor}[Thm]{Corollary}
\newreptheorem{Cor}{Corollary}

\newtheorem{Ques}[Thm]{Question}

\newtheorem{thm-int}{Theorem}

\theoremstyle{definition}
\newtheorem{Def-s}[Thm]{Definition}
\newtheorem{Def}[Thm]{Definition}
\newtheorem{Rem}[Thm]{Remark}

\newtheorem{Not}[Thm]{Notation}

\def\C{\ensuremath{\mathbb{C}}}

\def\H{\ensuremath{\mathbb{H}}}

\def\Q{\ensuremath{\mathbb{Q}}}
\def\R{\ensuremath{\mathbb{R}}}
\def\Z{\ensuremath{\mathbb{Z}}}

\def\cA{\ensuremath{\mathcal A}}

\def\cC{\ensuremath{\mathcal C}}
\def\cD{\ensuremath{\mathcal D}}
\def\cE{\ensuremath{\mathcal E}}
\def\cF{\ensuremath{\mathcal F}}

\def\cH{\ensuremath{\mathcal H}}

\def\cL{\ensuremath{\mathcal L}}

\def\cO{\ensuremath{\mathcal O}}
\def\cP{\ensuremath{\mathcal P}}

\def\cT{\ensuremath{\mathcal T}}

\def\nes{N\'eron--Severi }

\usepackage{todonotes}

\begin{document}

\title[Stability manifolds of varieties with finite Albanese morphisms]{Stability manifolds of varieties with finite Albanese morphisms}
\author{Lie Fu}
\address{L. F.:
Institute for Mathematics, Astrophysics and Particle Physics (IMAPP), Radboud University, PO Box 9010, 6500 GL, Nijmegen, Netherlands.}

\email{lie.fu@math.ru.nl}
\urladdr{https://www.math.ru.nl/$\sim$liefu/}

\author{Chunyi Li}
\address{C. L.:
Mathematics Institute, University of Warwick,
Coventry, CV4 7AL,
United Kingdom}
\email{C.Li.25@warwick.ac.uk}
\urladdr{https://sites.google.com/site/chunyili0401/}

\author{Xiaolei Zhao}
\address{X. Z.:
Department of Mathematics \\
South Hall 6607 \\
University of California \\
Santa Barbara, CA 93106, USA
}
\email{xlzhao@math.ucsb.edu}
\urladdr{https://sites.google.com/site/xiaoleizhaoswebsite/}


\begin{abstract} 
For a smooth projective complex variety whose Albanese morphism is finite, we show that every Bridgeland stability condition on its bounded derived category of coherent sheaves is geometric, in the sense that all skyscraper sheaves are stable with the same phase. Furthermore, we describe the stability manifolds of irregular surfaces and abelian threefolds with \nes rank one, and show that they are connected and contractible.
\end{abstract}
\date{\today}

\maketitle

\setcounter{tocdepth}{1}
\tableofcontents

\maketitle

\section{Introduction}

Let $X$ be a smooth projective variety over the field of complex numbers $\C$. Denote by $\Db(X)$ the bounded derived category of coherent sheaves on $X$. The notion of stability conditions on triangulated categories was introduced by Bridgeland in \cite{Bridgeland:Stab} (see Section \ref{subsect:GeneralityStabCond} for a recap).
Let $\Stab(X)$ be
the set of (full locally-finite numerical) stability conditions on $\Db(X)$. By the seminal result in \cite{Bridgeland:Stab}, $\Stab(X)$ is naturally endowed with a structure of complex manifold with local coordinates given by the central charge. We call $\Stab(X)$ the \emph{stability manifold} of $X$.

A complete description of the stability manifold has been worked out only for curves and abelian surfaces. More precisely, 
\begin{itemize}
    \item $\Stab(\mathbf{P}^1)\cong \C^2$,  see \cite{Okada:P1}.
    \item $\Stab(C)\cong \C\times \H$ for any smooth projective curve $C$ of genus $\geq 1$,  see \cite{Bridgeland:Stab,Macri:curves}; here $\H$ denotes the upper half-plane. 
    \item Abelian surfaces, see \cite{Bridgeland:K3,HMS:generic_K3s}.
\end{itemize}

In any of the following cases, a principal connected component of the stability manifold is determined, which is expected to be the whole stability manifold:
\begin{itemize}
  \item K3 surfaces with Picard rank one, see \cite{Bridgeland:K3,K3Pic1}.
  \item The projective plane $\mathbf P^2$, see \cite{P2stab}.
  \item Abelian threefolds with \nes rank one, see \cite{Dulip-Antony:I,Dulip-Antony:II, BMS:stabCY3s}.
\end{itemize}
Moreover, in each case above, the stability manifold (or the principal component) is simply connected and contains an open subset consisting of the so-called \textit{geometric} stability conditions, meaning that all skyscraper sheaves are stable with the same phase (see Definition \ref{def:geostab}). Note that the condition of equal phase turns out to be automatic, see Proposition \ref{prop:skyscrapersamephase}.

In this paper, we provide more instances of algebraic varieties whose stability manifolds have the same feature as the aforementioned examples. Our first main result reads as follows:

\begin{Thm}[{Corollary \ref{cor:albfiniteimplygeom}}]\label{thm:main-1}
Let $X$ be a connected smooth projective variety over $\C$. If the Albanese morphism of $X$ is finite, then every numerical stability condition on $\Db(X)$ is geometric.
\end{Thm}

Here we briefly recall some examples of varieties with finite Albanese morphisms. The most basic ones are abelian varieties and curves of genus $\geq 1$. To produce new examples out of old ones, observe that this property is stable under products and is inherited by subvarieties and finite ramified covers. See \cite{Kawamata-AV} for more on the structure of a variety with finite Albanese morphism.

Also, note that a smooth projective surface $S$ of \nes rank one has finite Albanese morphism if and only if it is \textit{irregular}, that is, $q(S):=\dim H^1(S, \mathcal{O}_S)\neq 0$. One expects that a generic minimal surface of irregularity $q\geq 2$ is of this type. Here are some geometrically interesting examples: the Fano surface of lines on a generic cubic threefold (\cite{Roulleau:fanosurface}), the Fano surface of planes on a generic cubic fivefold.

\begin{Thm}[Corollary \ref{cor:irregsurf}]
\label{thm:main-2}
The stability manifold of an irregular surface of \nes rank one is connected and contractible. 
\end{Thm}
One should compare Theorem \ref{thm:main-2} with the case of abelian surfaces: as the derived category of an abelian surface has no spherical objects (see \cite[Lemma 15.1]{Bridgeland:K3}), \cite[Theorem 1]{HMS:generic_K3s} says in particular that the stability manifold of an abelian surface is connected and simply connected. 

\vspace{0.5cm}

Finally, we establish the following result on abelian threefolds.
Recall that for an $n$-dimensional polarized smooth projective variety $(X,H)$, $\Stab_H(X)$ denotes the stability manifold with respect to the surjection $\mathrm{K}(X)\twoheadrightarrow \Lambda_H$, where $\Lambda_H$ is the image of the map $ \mathrm{K}(X)\to \mathbb{R}^n$ sending a class $[E]$ to the vector $(H^n\rk(E),H^{n-1}\ch_1(E),H^{n-2}\ch_2(E),\dots,\ch_n(E))$.

\begin{Thm}[Corollary \ref{cor:ab3space}]
\label{thm:main-3}
Let $(A,H)$ be a polarized abelian threefold, then $\Stab_H(A)=\tilde {\mathfrak P}$ as that constructed in \cite[Theorem 9.1]{BMS:stabCY3s}. In particular, when $A$ has \nes rank one, $\tilde {\mathfrak P}$ is the whole stability manifold and it is connected and contractible.
\end{Thm}
This completes the result of \cite{Dulip-Antony:I,Dulip-Antony:II, BMS:stabCY3s} on the stability manifold of abelian threefolds of \nes rank one.

\subsection*{Sketch of the proof} Theorem \ref{thm:main-1} mainly relies on the result in \cite{Polishchuk:families-of-t-structures} and the uniqueness of Harder--Narasimhan filtration. By \cite[Corollary 3.5.2]{Polishchuk:families-of-t-structures}, the natural action of the group Pic$^0(X)$ on the stability manifold is trivial. Hence the Harder--Narasimhan factors of every Pic$^0(X)$-invariant object, e.g. a skyscraper sheaf, must also be Pic$^0(X)$-invariant. It follows by the assumption on the Albanese morphism that such factors of a skyscraper sheaf may only be  skyscraper sheaves. In particular, all skyscraper sheaves are stable.
On the other hand, we prove in general that all skyscraper sheaves are of the same phase when they are all stable, see Proposition \ref{prop:skyscrapersamephase}. This seemingly-simple statement is probably known to some experts, but we found no proof in the literature. 

Theorem \ref{thm:main-2} on irregular surfaces then follows from the same observation as that in \cite[Section 10]{Bridgeland:K3}, which implies that every geometric stability condition can be constructed via tilting heart up to the $\glt$-action. 

The proof of Theorem \ref{thm:main-3} on abelian threefolds is more involved. First, for any abelian variety $A$, using a variant of the Fourier--Mukai transform between  $\Db(A)$ and $\Db(A^\vee)$, we show that all simple semi-homogeneous vector bundles are stable with respect to all stability conditions, see Corollary \ref{cor:EpqStable}. Second, specializing to abelian threefolds, we show that every stability condition is with the same central charge as that constructed in \cite{BMS:stabCY3s}. The actual difficulty of the whole argument is to prove that their heart structures are also the same. Proposition \ref{prop:ab3phasebound}, Lemma \ref{lem:distsym} and Lemma \ref{lem:samecentralcharge} are to deal with this issue. Roughly speaking, we show that there are enough morphisms from simple semi-homogeneous vector bundles to other stable objects so that a stability condition can be uniquely determined by the phases of all semi-homogeneous vector bundles and the central charge.

\subsection*{Acknowledgment.} We are very grateful to Arend Bayer,  Emanuele Macr\`i, and Paolo Stellari for pointing out references and enlightening comments. In particular, the shortcut proof of Lemma \ref{lem:arend} is due to Arend Bayer. We also thank the referees for helpful comments and references. L.~F.~ is supported by the Agence Nationale de la Recherche (ANR) under project numbers ANR-20-CE40-0023 and ANR-16-CE40-0011, he was also supported by the Radboud Excellence Initiative program. C.~L.~ is a University Research Fellow supported by the Royal Society URF$\backslash$R1$\backslash$201129 “Stability condition and application in algebraic geometry”.  X.~Z.~ is partially supported by the Simons Collaborative Grant 636187.

\section{Geometric stability conditions}
\subsection{Generalities on stability conditions}
\label{subsect:GeneralityStabCond}
Let $X$ be a smooth projective variety over $\C$. Denote by $\Db(X)$ the bounded derived category of coherent sheaves on $X$. We recall some basic notions of stability conditions.

\begin{Def}
A \emph{slicing} $\cP$ on $\Db(X)$ is a collection of full additive subcategories $\cP(\phi)\subset \Db(X)$ indexed by all $\phi\in \R$ such that
\begin{enumerate}
    \item $\cP(\phi)[1]=\cP(\phi+1)$;
    \item if $\phi_1>\phi_2$ and $F_i\in \operatorname{Obj}(\cP(\phi_i))$, then $\Hom(F_1,F_2)=0$;
    \item \label{c} for any object $E$ in $\Db(X)$, there are real numbers $\phi_1>\dots>\phi_m$, objects $E_i$ in $\Db(X)$, and a collection of distinguished triangles 
    \begin{equation*}
        \begin{tikzcd}[column sep=tiny]
0=E_0  \arrow[rr]& & E_1 \arrow[dl] \arrow[rr]& & E_2\arrow[r]\arrow [dl]& \cdots\arrow[r] &E_{m-1}\arrow [rr]& & E_m=E\arrow[dl]\\
& A_1 \arrow[ul,dashed, "+1"] && A_2\arrow[ul,dashed, "+1"] &&&& A_m \arrow[ul,dashed, "+1"]
\end{tikzcd}
    \end{equation*}
such that  $A_i=\mathrm{Cone}(E_{i-1}\rightarrow E_i)$ is an object in $\cP(\phi_i)$ for every $1\leq i\leq m$.
\end{enumerate}
\label{def:slicing}
\end{Def}
\begin{Rem}\label{rem:slicing}
Non-zero objects of $\cP(\phi)$ are called \emph{semistable} with phase $\phi$; simple objects of $\cP(\phi)$  are called
\emph{stable}. 
The sequence of maps in \eqref{c} is called the \emph{Harder--Narasimhan (HN) filtration} of $E$. It is not hard to see that the HN filtration is unique up to isomorphism.  We will call each $A_i$ the \emph{HN factor} of $E$. We denote $\phi_\cP^+(E)\coloneqq \phi_1$ and $\phi_\cP^-(E)\coloneqq \phi_m$; when $E$ is semistable, we denote $\phi_\cP(E)$ the phase of $E$ (or simply $\phi(E)$ when there is no ambiguity).
\end{Rem}

\begin{Not}
Let $\cP$ be a slicing. For every interval $I\subset \R$, we define a \textit{truncated HN factor} $\mathrm{HN}_{\cP}^I(E)$ of $E$ as follows. In practice, we would also write  $\leq \phi$ or $>\phi$ for the interval $(-\infty,\phi]$ or $(\phi,+\infty)$ respectively when it causes no confusion.

In the HN filtration \ref{def:slicing}\eqref{c} of $E$, let $a,b\in \{1,2,\dots,m\}$ be such that $\phi_i\in I$ when and only when $a\leq i\leq b$, then $\mathrm{HN}_\cP^I(E)$ is defined to be $\mathrm{Cone}(E_{a-1}\rightarrow E_b)$ (and to be $0$ if no $\phi_i$ is contained in $I$). We denote by $\cP(I)$ the full subcategory of all objects $E\in\Db(X)$ such that $\phi^\pm_{\cP}(E)\in I$.
\label{not:hntrunc}
\end{Not}

Denote by $\mathrm{K}(X)$ the Grothendieck group of $\Db(X)$.

\begin{Def}
A \emph{Bridgeland pre-stability condition} on $\Db(X)$ is a pair $\sigma=(\cP,Z)$, where 
\begin{itemize}
    \item $\cP$ is a slicing of $\Db(X)$;
    \item $Z:\mathrm{K}(X)\rightarrow \C$ is a group homomorphism, called the \emph{central charge};
\end{itemize}
such that for any non-zero object $E$ in $\cP(\phi)$, we have $Z([E])=m(E)e^{i\pi \phi}$ for some $m(E)\in \mathbb{R}_{>0}$.
\label{def:prestab}
\end{Def}

The set of all pre-stability conditions has a natural topology induced by the following generalised metric function: for two pre-stability conditions  $\sigma_1=(\cP_1, Z_1)$ and $\sigma_2=(\cP_2, Z_2)$, their distance is defined as
$$\mathrm{dist}(\sigma_1,\sigma_2)\coloneqq\sup_{\substack{E\in\Db(X)\\ E\neq 0 }}\left\{\abs{\phi_{\cP_1}^-(E)-\phi_{\cP_2}^-(E)},\abs{\phi_{\cP_1}^+(E)-\phi_{\cP_2}^+(E)}, \left\|Z_1-Z_2\right\|\right\}\in [0,+\infty].$$

We have the following two natural group actions on the set of pre-stability conditions; see \cite[Remark 5.14]{Emolo-Benjamin:lecture-notes} for more details.
\begin{itemize}
     \item The universal covering $\glt$ of the group $\GL_2^+(\mathbb{R})$ acts on the right of this set. This $\glt$-action does not affect the  (semi)stability of any object.
     \item The group $\Aut(\Db(X))$ of autoequivalences of $\Db(X)$ acts on the left of this set.
\end{itemize}

 Let  $\Lambda$ be a free abelian group of finite rank and let $\lambda\colon \mathrm{K}(X)\twoheadrightarrow \Lambda$ be  a surjective homomorphism. A typical choice (by default) for $\Lambda$ is the numerical Grothendieck group $\Kn(X)$ and $\lambda$ is the natural projection $\operatorname{K}(X)\twoheadrightarrow \Kn(X)$. 

\begin{Def}[\cite{Bridgeland:Stab, Kontsevich-Soibelman:stability}]
We say that a pre-stability condition $(\cP,Z)$  satisfies the \emph{support property} (with respect to $\Lambda$, or rather $\lambda$) if the central charge $Z$ factors through $\lambda\colon\mathrm{K}(X)\twoheadrightarrow \Lambda$ and there is  a quadratic form $Q_\Lambda$ on $\Lambda\otimes \mathbb R$ such that 
\begin{enumerate}
    \item the kernel $\Ker Z\subset \Lambda\otimes \mathbb R$ of the central charge is negative definite with respect to $Q_\Lambda$;
    \item for any semistable object $E$, we have $Q_{\Lambda}(E)\geq 0$.
\end{enumerate}
A pre-stability condition satisfying the support property is called a \emph{stability condition}, and we denote the set of all stability conditions  (with respect to $\Lambda$, or rather $\lambda$) as $\Stab_{\Lambda}(X)$.

When $\lambda$ factors through $\mathrm{K}(X)\rightarrow \mathrm{K}_{\mathrm{num}}(X)$, we will call the stability condition \emph{numerical}. 
\label{def:supportproperty}
\end{Def}

\begin{Thm}[{\cite[Theorem 7.1]{Bridgeland:Stab}, \cite[Theorem 1.2]{Arend:shortproof}}]
 The space Stab$_\Lambda(X)$ of stability conditions
is naturally a complex manifold of dimension $\rank(\Lambda)$, such that the map forgetting the slicing
$$\pi_Z: \mathrm{Stab}_\Lambda(X) \rightarrow \mathrm{Hom}_{\mathbb Z}(\Lambda, \mathbb C) \ \ \ \sigma =(Z,\cP) \mapsto Z 
$$
is a local biholomorphic isomorphism at every point of $\Stab_\Lambda(X)$.
\label{thm:spaceasamfd}
\end{Thm}

\begin{Rem}
Note that the notion of pre-stability condition in this paper (as well as many others) is called stability condition in Bridgeland's original paper \cite{Bridgeland:Stab}. 

When the lattice $\Lambda$ is the numerical Grothendieck group, our definition of stability conditions is the same as that of the full locally-finite numerical stability conditions in \cite{Bridgeland:K3}. We denote the space of such stability conditions as $\Stab(X)$ and call it the \textit{stability manifold} of $X$. 
\end{Rem}

\subsection{Geometric stability conditions}
\begin{Def}\label{def:geostab}
A stability condition $\sigma$ on $\Db(X)$ is called \emph{geometric} (with respect to $X$) if for each point $p\in X$, the skyscraper sheaf $\cO_p$ is   $\sigma$-stable, and all skyscraper sheaves are of the same phase.
\end{Def}
The definition is similar to that of \cite[Definition 10.2]{Bridgeland:K3}. The only difference is that in \cite{Bridgeland:K3}, the author further assumes a `goodness' condition. In the next section
, we will see that the `goodness' condition is redundant at least in the surface case.
In the next proposition, we show that the condition of having the same phase is redundant for numerical stability conditions when $X$ is a smooth connected projective variety. The result might be known to some experts, but there is no proof written down as far as we know.

\begin{Prop}\label{prop:skyscrapersamephase}
Let $X$ be a connected smooth projective variety. Let $\sigma$ be a numerical stability condition on $\Db(X)$ such that every skyscraper sheaf is $\sigma$-stable. Then all skyscraper sheaves are of the same phase, in other words, $\sigma$ is geometric.
\end{Prop}

We first establish the following technical lemma.

\begin{Lem}\label{lem:HNfactorcontrol}
Let $\sigma$ be a stability condition on $\Db(X)$. Let $g\colon F\rightarrow E$ be a morphism and  $\tilde F\coloneqq\mathrm{Cone}(F\xrightarrow{g} E)[-1]$ in $\Db(X)$.
\begin{enumerate}
    \item[1.] Then $\mathrm{HN}_\sigma^{a}(F)\cong \mathrm{HN}_\sigma^{a}(\tilde F)$ for every $a\notin [\phi^-(E)-1,\phi^+(E)]$.
    \item[2.] If $E$ is non-zero and the second smallest phase of the HN factors of $E$ is greater than $\phi^-(E)+1$, then $\mathrm{HN}_\sigma^{\leq \phi^-(E)}(F)\not\cong\mathrm{HN}_\sigma^{\leq \phi^-(E)}(\tilde F)$.
\end{enumerate}
\end{Lem}
\begin{proof}
As for the first statement, we consider the following diagram of distinguished triangles:
\begin{center}
	\begin{tikzcd}
		0\arrow{d}\arrow{r}&\mathrm{HN}^{> \phi^+(E)}_\sigma(\tilde F) \arrow{d}{\mathrm{can}} \arrow{r}{\id}
		& \mathrm{HN}^{> \phi^+(E)}_\sigma(\tilde F) \arrow{d}{f}\arrow{r}& 0\arrow{d}\\
		E[-1]\arrow{r} \arrow{d}{\id}& \tilde F \arrow{r}{g'}\arrow{d} & F\arrow{r}{g}\arrow{d} & E\arrow{d}{\id}\\
		E[-1]\arrow{r} & \mathrm{HN}^{\leq \phi^+(E)}_\sigma(\tilde F)\arrow{r} & K\arrow{r} & E
	\end{tikzcd}
\end{center}
The morphism $f$ is the composition of $g'$ and $\mathrm{can}$. The square on the top-right commutes since $g\circ f\in\Hom(\mathrm{HN}^{> \phi^+(E)}_\sigma(\tilde F),E)=0$. The sequence on the bottom is the cone of distinguished triangles on the top and middle, hence it is a distinguished triangle by the octahedral axiom.

For any $\sigma$-semistable object $A$ with phase $>\phi^+(E)$, by applying $\Hom(A,-)$ to the distinguished triangle on the bottom, we have $\Hom(A,K)=0$. Therefore, the object $K$ is in $\cP_\sigma(\leq \phi^+(E))$. 

Note that $K=\mathrm{Cone}(\mathrm{HN}^{> \phi^+(E)}_\sigma(\tilde F) \xrightarrow{f} F)$. It follows $K\cong \mathrm{HN}^{\leq  \phi^+(E)}_\sigma(F) $ and $$\mathrm{HN}^{> \phi^+(E)}_\sigma(\tilde F)\cong \mathrm{HN}^{> \phi^+(E)}_\sigma(F).$$

By a similar argument for the diagram
\begin{center}
	\begin{tikzcd}
		E[-1]\arrow{d}{\id}\arrow{r}& L \arrow{d}{} \arrow{r}
		& \mathrm{HN}^{\geq \phi^-(E)-1}_\sigma(F) \arrow{d}{}\arrow{r}& E\arrow{d}{\id}\\
		E[-1]\arrow{r} \arrow{d}& \tilde F \arrow{r}\arrow{d} & F\arrow{r}{}\arrow{d}{\mathrm{can}} & E\arrow{d}{}\\
		0\arrow{r} & \mathrm{HN}^{< \phi^-(E)-1}_\sigma(F)\arrow{r}{\id} & \mathrm{HN}^{< \phi^-(E)-1}_\sigma(F)\arrow{r} & 0,
	\end{tikzcd}
\end{center}
we get $\mathrm{HN}^{< \phi^-(E)-1}_\sigma(\tilde F)\cong \mathrm{HN}^{< \phi^-(E)-1}_\sigma(F).$ The first  statement holds.\\

As for the second statement, note that $\cA\coloneqq\cP_\sigma((\phi^-(E)-1,\phi^-(E)])$ is a heart on $\Db(X)$. Apply $\cH^\bullet_\cA(-)$ to the distinguished triangle $\tilde F\rightarrow F\xrightarrow{g} E$, we get the long exact sequence 
$$\cH^{-1}_\cA(E)\rightarrow \cH^{0}_\cA(\tilde F)\rightarrow \cH^{0}_\cA( F)\rightarrow \cH^{0}_\cA(E)\rightarrow \cH^{1}_\cA(\tilde F)\xrightarrow{h} \cH^{1}_\cA(F)\rightarrow \cH^{1}_\cA(E).$$
Note that for every object $B$, $\cH^i_\cA(B)=\mathrm{HN}_\sigma^{(\phi^-(E)-i-1,\phi^-(E)-i]}(B)[i]$. By the assumption on $E$, we have $\cH^{\pm 1}_\cA(E)=0$ and $\cH^0_\cA(E)=\mathrm{HN}^-(E):=\mathrm{HN}_\sigma^{\phi^-(E)}(E)$.

If the map $h$ is not an isomorphism, then  the kernel of $h$ in $\cA$ is a  quotient object of $\mathrm{HN}^-(E)$, which must be in $\cP_\sigma(\phi^-(E))$. We get a short exact sequence
$$0\rightarrow \ker h\rightarrow \mathrm{HN}_\sigma^{\phi^-(E)-1}(\tilde F)[1]\xrightarrow{h}\mathrm{HN}_\sigma^{\phi^-(E)-1}(F)[1]\rightarrow 0,$$ which implies $\mathrm{HN}_\sigma^{\phi^-(E)-1}(\tilde F)\not\cong\mathrm{HN}_\sigma^{\phi^-(E)-1}(F)$ as $[\ker h]$ is not $0$ in $\Kn(X)$. 

If the map $h$ is an isomorphism, then we have the short exact sequence $$0\rightarrow \cH^{0}_\cA(\tilde F)\rightarrow \cH^{0}_\cA( F)\rightarrow \cH^{0}_\cA(E)\rightarrow 0 .$$
Since $[\cH^{0}_\cA(E)]\neq 0$ in $\Kn(X)$, we have $$\mathrm{HN}_\sigma^{(\phi^-(E)-1,\phi^-(E)]}(\tilde F)=\cH^0_\cA(\tilde F)\not\cong \cH^0_\cA(F)=\mathrm{HN}_\sigma^{(\phi^-(E)-1,\phi^-(E)]}(\tilde F).$$
In any case, we always have $\mathrm{HN}_\sigma^{[\phi^-(E)-1, \phi^-(E)]}(F)\not\cong\mathrm{HN}_\sigma^{[\phi^-(E)-1, \phi^-(E)]}(\tilde F)$.
\end{proof}

\begin{proof}[{Proof of Proposition \ref{prop:skyscrapersamephase}}]
For a smooth connected curve $C$ on $X$, under the assumption that all skyscraper sheaves of points on the curve $C$ are $\sigma$-stable, but suppose that not all of them are with the same phase. Then there exist a large enough integer $m$ and points $p_1,\dots,p_m$, $q_1,\dots,q_m$ on $C$ such that 
\begin{itemize}
    \item $\cO_C(-p_1-\dots-p_m)\cong \cO_C(-q_1-\dots-q_m)\eqqcolon\cL$;
    \item $\phi(\cO_{p_1})<\phi(\cO_{q_i})$ for every $1\leq i\leq m$.
\end{itemize}
(Note that we used the assumption that all skyscraper sheaves of points are stable to be able to talk about their phases.)
In particular, as the central charge factors via $\operatorname{K}_{\mathrm{num}}(X)$, $\phi(\cO_{p_1})\leq  \phi(\cO_{q_i})- 2$ for every $i$.  We may also assume $\phi(\cO_{p_1})=\min\{\phi(\cO_{p_i})\}$.

Note that $\cL$ has two expressions $\mathrm{Cone}(\cO_C\rightarrow \cO_{p_1}\oplus\dots\oplus \cO_{p_m})[-1]$ and $\mathrm{Cone}(\cO_C\rightarrow \cO_{q_1}\oplus\dots\oplus \cO_{q_m})[-1]$. By Lemma \ref{lem:HNfactorcontrol}.1, the HN factor $\mathrm{HN}^a$ of  $\cL$ is isomorphic to that of $\cO_C$ for every $a< \min\{\phi(\cO_{q_i})\}-1$. By Lemma \ref{lem:HNfactorcontrol}.2, $\mathrm{HN}^{\leq \phi(\cO_{p_1})}(\cL)\not\cong \mathrm{HN}^{\leq \phi(\cO_{p_1})}(\cO_C)$, which leads to a contradiction.

Hence, the phase of all skyscraper sheaves of points on $C$ are the same. As $X$ is connected and projective, this implies that all skyscraper sheaves are of the same phase.
\end{proof}

The following lemma is an immediate generalization of {\cite[Lemma 10.1]{Bridgeland:K3}} in the higher dimensional case.
\begin{Lem}[{\cite[Lemma 10.1]{Bridgeland:K3}}]\label{lem:geohearts}
Let $X$ be an $n$-dimensional smooth projective variety and $\sigma=(\cP,Z)$ be a geometric stability condition on $X$ such that $\cO_p\in\cP(1)$. Let $F$ be an object of $\Db(X)$. Then
\begin{enumerate}
    \item if $F\in\cP((0,1])$, then $\cH^i(F)$ vanishes unless $-n+1\leq i \leq 0$, and moreover, $\cH^{-n+1}(F)$ is torsion-free;
    \item if $F\in\Coh(X)$, then $F\in \cP((-n+1,1])$; if $F$ is a torsion sheaf, then $F\in\cP((-n+2,1])$.
\end{enumerate}
\end{Lem}
\begin{proof}
\emph{(a)} One may assume that $F$ is stable and is not a skyscraper sheaf. For any point $p\in X$, since $\cO_p$ is stable with phase $1$,  $\Hom(\cO_p,F[i])=\Hom(F,\cO_p[i-1])=0$ for all $i\leq 0$. By Serre duality, $\Hom(F,\cO_p[i])=0$ unless $0\leq i\leq n-1$.

It follows from  \cite[Proposition 5.4]{Bridgeland-Maciocia:K3Fibrations} that $F$ is quasi-isomorphic to a length $n$ complex of locally free sheaves. So $F$ satisfies the condition in part \emph{(a)}.\\

\emph{(b)} Let $F$ be a coherent sheaf. For every object $E\in\cP(>1)$, by part \emph{(a)}, $\cH^i(E)=0$ when $i\geq 0$. So $\Hom(E,F)=0$. For every object $G\in \cP(\leq -n+1)$, by part \emph{(a)}, $\cH^i(G)=0$ when $i\leq 0$. So $\Hom(F,G)=0$. It follows that $F\in\cP((-n+1,1])$. 

Now let $F$ be a torsion sheaf. We have a distinguished triangle
$$E\rightarrow F\rightarrow G\xrightarrow{+1}$$
for some $E\in\cP((-n+2,1])$ and $G\in\cP((-n+1,-n+2])$. By part \emph{(a)}, $\cH^i(G)=0$ unless $0\leq i\leq n-1$, and $\cH^{0}(G)$ is torsion-free. Since $F$ is a torsion sheaf, $\Hom(F,G)=0$. This can only happen when $G=0$, so $F\in \cP((-n+2,1])$.
\end{proof}

\subsection{Stability of skyscraper sheaves}

The following basic result is used in the study of homogeneous vector bundles on abelian varieties.  
\begin{Lem}
\label{lem:FiniteSupport}
Let $A$ be an abelian variety and $E\in \Db(A)$. Further assume that for any $\xi\in A^\vee$, we have that $E\otimes P_\xi\simeq E$, where $P_\xi$ is the line bundle on $A$ parameterized by $\xi$. Then the support of $E$ is finite.
\end{Lem}
\begin{proof}
See for example the proof of \cite[Proposition 11.8]{Polishchuk:book}.
\end{proof}

\begin{Thm}
Let $X$ be a smooth projective variety such that its Albanese morphism is finite, then skyscraper sheaves are all stable with respect to every numerical stability condition on $\Db(X)$.
\label{thm:alwaysgeostab}
\end{Thm}
\begin{proof}
Fix an arbitrary numerical stability condition on $\Db(X)$. Let $p\in X$ be a closed point. Let $E_1, \dots, E_n$ be the Jordan--H\"older factors in the Harder--Narasimhan filtration associated with $\cO_p$. By Polishchuk \cite[Corollary 3.5.2]{Polishchuk:families-of-t-structures}, the group $\Pic^0(X)\subset \Aut(\Db(X))$ acts trivially on the stability manifold. In particular, the Harder--Narasimhan filtration, as well as the Jordan--H\"older factors, are preserved by the action of $\Pic^0(X)$. As a consequence, each $E_i$ satisfies the condition that $E_i\otimes L\simeq E_i$ for any $L\in \Pic^0(X)$.

Let $a:X\rightarrow \Alb(X)$ be the Albanese morphism. For any $\xi\in \Pic^0(\Alb(X))$, denote by $P_\xi$ the corresponding line bundle on $\Alb(X)$. Then for any $i$, we have $E_i\otimes a^*(P_\xi)\simeq E_i$. By the projection formula, 
$$Ra_*(E_i)\otimes P_\xi \simeq Ra_*(E_i),$$
which implies that the support of $Ra_*(E_i)$ is finite by Lemma \ref{lem:FiniteSupport}. Combined with the assumption that $a$ is a finite morphism, we see that the support of $E_i$ is also finite, for any $i$.

The theorem then follows by Lemma \ref{lem:arend}.
\end{proof}

\begin{Lem}\label{lem:arend}
Let $X$ be a smooth variety and $\sigma$ a stability condition on $\Db(X)$. Assume that the Jordan--H\"older factors of $\cO_p$ have finite support. Then $\cO_p$ is stable.
\end{Lem}
\begin{proof}

Suppose $\cO_p$ is not stable, then there exists a stable object $E\neq \cO_p$ supported at $p$.  Let $k, l$ be the maximal and minimal non-vanishing cohomology degrees of $E$. Then there exists composition of morphisms: $$E \xrightarrow{\mathrm{can}} \cH^k(E)[-k] \xrightarrow{\iota_1} \cO_p[-k] \xrightarrow{\iota_2} \cH^l(E)[-k] \xrightarrow{\mathrm{can}} E[l-k].$$
As both $\cH^k(E)$ and $\cH^l(E)$ are supported at $p$, we may let $\iota_i$'s be non-zero. Their composition  $\iota_2\circ\iota_1$ is then non-zero as well.

The whole composition is non-zero because it induces a non-zero map from the term of $E$ with maximal cohomology degree to the term of $E[l-k]$ with minimal cohomology degree. Since $k \geq l$, by the stability of $E$, we may only have $k=l$, and $E=\cH^k(E)[-k]=\cO_p[-k]$ which is stable. This leads to a contradiction.
\end{proof}

\begin{Cor}\label{cor:albfiniteimplygeom}
Let $X$ be a connected smooth projective variety such that its Albanese morphism is finite. Then all numerical stability conditions on $\Db(X)$ are geometric.
\end{Cor}
\begin{proof}
It is a combination of Proposition \ref{prop:skyscrapersamephase} and Theorem \ref{thm:alwaysgeostab}.
\end{proof}

Let $A$ be an abelian variety. For every element $\frac{[L]}l\in \mathrm{NS}(A)\otimes_\Z \Q$, there exist  simple semi-homogeneous vector bundles with $\frac{\mathrm{det}}{\rk}=\frac{[L]}l$  in the sense of  Mukai \cite{Mukai:semi-homogeneous}, see \cite[Section 4]{Orlov:abelian_equivalences}. The following result will be useful in studying the stability manifold of abelian threefold in Section 4.

\begin{Cor}[cf. {\cite[Proposition 3.1.4]{Polishchuk:LIobjects}}]
\label{cor:EpqStable}
Let $A$ be an abelian variety. Then all simple semi-homogeneous vector bundles are stable with respect to any numerical stability condition on $\Db(A)$. Moreover, given a numerical stability condition, all simple semi-homogeneous vector bundles with $\frac{\mathrm{det}}{\rk}=\frac{[L]}l$ are with the same phase.
\end{Cor}

\begin{proof}
Denote by $M$ the moduli space parameterizing all simple semi-homogeneous vector bundles with $\frac{\mathrm{det}}{\rk}=\frac{[L]}l$. In particular, $M\cong A^\vee$ by \cite[Theorem 7.11]{Mukai:semi-homogeneous}. Denote by $\cE_{\frac{[L]}l}$ the universal family on $A\times M$. By \cite[Lemma 4.8]{Orlov:abelian_equivalences}, and \cite[Theorem 5.1 and 5.4]{Bridgeland:EqFMT},  the Fourier--Mukai transformation $\Phi_{\cE_{\frac{[L]}l}}$ is an equivalence between $\Db(A)$ and $\Db(M)$.

Note that the property of being stable with respect to any stability condition is preserved under any equivalence. The skyscraper sheaves on $M$ are mapped to all simple semi-homogeneous vector bundles with $\frac{\mathrm{det}}{\rk}=\frac{[L]}l$. The statement now follows from  Theorem \ref{thm:alwaysgeostab}. 
\end{proof}

Note that the stability of these objects is already proved in \cite[Proposition 3.1.4]{Polishchuk:LIobjects}. 

\section{Surface case}
Let $(X,H)$ be a polarized smooth projective variety. We fix the following surjection from $\operatorname{K}(X)$, whose image, denoted by $\Lambda_H$, is clearly a lattice in $\mathbb{R}^n$: 
$$\lambda_H:\mathrm{K}(X)\twoheadrightarrow \Lambda_H; [E]\mapsto (H^n\rk(E),H^{n-1}\ch_1(E),H^{n-2}\ch_2(E),\dots,\ch_n(E)),$$
where $n$ is the dimension of $X$. We denote the set of all (resp. geometric) stability conditions with respect to $\Lambda_H$ as $\Stab_H(X)$ (resp. $\Stab^{\mathrm{Geo}}_H(X)$). In particular, if $S$ is a surface of \nes rank one, then $\Stab_H(S)$ is $\Stab(S)$, the whole stability manifold of $S$. 
\subsection{Le Potier function}
Recall that the $H$-slope of a coherent sheaf $F$ on $X$ is defined as 
$$\mu_H(F):=\begin{cases}\frac{H^{n-1}\ch_1(F)}{H^n\rk(F)}, & \text{ if $\rk(F)>0$;}\\
+\infty, & \text{ if $\rk(F)=0$.}\end{cases}$$
A coherent sheaf $F$ is called $\mu_H$-(semi)stable if for every proper non-zero subsheaf $E$, one has 
$$\mu_H(E)<(\leq)\mu_H(F/E).$$

\begin{Def}\label{def:lpfunction}
Let $(S,H)$ be a polarized surface. We define the \emph{Le Potier function} $\Phi_{S,H}:\mathbb R\rightarrow \mathbb R$ as:
$$\Phi_{S,H}(x)\coloneqq \overline{\lim_{\mu\to x}}\sup_{F\in\Coh(X)}\left\{\frac{\ch_2(F)}{H^2\rk(F)}\middle|\text{ $F$ is $\mu_H$-semistable with }\mu_H(F)=\mu\right\}.$$
\end{Def}
\begin{Prop}\label{defp:lpf}
The Le Potier function is well-defined satisfying $\Phi_{S,H}(x)\leq \frac{x^2}2.$ It is the 
 smallest upper semi-continuous function satisfying 
$$\frac{\mathrm{ch}_2(F)}{H^2\mathrm{rk}(F)}\leq \Phi_{S,H}\left(\frac{H\mathrm{ch}_1(F)}{H^2\mathrm{rk}(F)}\right)$$
for every torsion-free $\mu_H$-stable (or semistable) sheaf $F$.
\end{Prop}
\begin{proof}
By \cite[Theorem 5.2.5]{HL:Moduli}, for every rational number $\mu$, there exists a $\mu_H$-stable sheaf $F$ with $\mu_H(F)=\mu$. The value of the function at every point is therefore in $\R\cup\{+\infty\}$.

The Bogomolov inequality states that 
$$2\rk(F)\mathrm{ch}_2(F)\leq  \mathrm{ch}_1(F)^2.$$
for every $\mu_H$-semistable sheaf $F$. Combined with the Hodge Index Theorem, we have 
$$2H^2\rk(F)\mathrm{ch}_2(F)\leq  (H\mathrm{ch}_1(F))^2$$
for every $\mu_H$-semistable sheaf $F$. When $F$ is torsion-free, we can divide both sides by $(H^2\rk(F))^2$, which implies $\Phi_{S,H}(x)\leq \frac{x^2}2$. In particular, the function is well-defined, i.e. valued in $\mathbb{R}$.

The last statement follows directly from the definition of $\Phi_{S,H}$.
\end{proof}

\begin{Rem}
The Le Potier function is only known for very few polarized surfaces. When $S$ is an abelian surface, it is known that $\Phi_{S,H}(x)=\frac{x^2}2$. When $S$ is the projective plane, the function is known thanks to the work \cite{Drezet-LePotier}. The explicit formula for $\Phi_{\mathbf{P}^2}$ is very complicated, see \cite{P2stab} for more details. The function is also known or partially known for polarized K3 surfaces, del Pezzo surfaces, and a few sporadic surfaces like the intersection of a quadric and a quintic $S_{2,5}$ in $\mathbf{P}^4$; see \cite{Zhiji-bghypersurfaces, quintic} for more details. Very recently, Lahoz and Rojas showed in \cite[Example 2.12(2)]{LahozRojaz} that $\Phi_{S,H}(x)=\frac{x^2}2$ for any smooth projective surface with finite Albanese morphism. 
\end{Rem}

\subsection{Tilting construction} In this section, we recall the tilting construction of  stability conditions on polarized surfaces.  We refer to the lecture notes \cite[Section 6]{Emolo-Benjamin:lecture-notes} for more details, and to \cite{Aaron-Daniele} for the original treatment. The same construction gives weak stability conditions on threefolds or higher dimensional varieties. We will only use the threefold case in the next section. Readers interested in more details are referred to \cite{BMT:3folds-BG, BBMT:Fujita}.

Let $(X,H)$ be a polarized variety. For every $\beta\in \R$, we define a pair of subcategories:
\begin{align*}
    \cT_\beta \coloneqq \{F\in \Coh(X)|\text{ any $\mu_H$-semistable factor of } F \text{ satisfies } \mu_H(F)>\beta\}; \\ \cF_\beta \coloneqq \{F\in \Coh(X)|\text{ any $\mu_H$-semistable factor of } F \text{ satisfies } \mu_H(F)\leq \beta\}.
\end{align*}

This is a torsion pair on $\Coh(X)$. We denote the tilted heart as
\[\Coh^\beta(X)\coloneqq \langle \cT_\beta,\cF_\beta[1]\rangle=\left\{F\in \Db(X)\middle| \begin{aligned}  \cH^i(F)=0 \text{ for } i\neq 0,-1;\\
\cH^0(F)\in \cT_\beta;\\
\cH^{-1}(F)\in \cF_\beta.\end{aligned} \right\}.\]

Recall from \cite[Lemma 5.11]{Emolo-Benjamin:lecture-notes} that giving a stability condition $(\cP,Z)$ on a triangulated category is equivalent to giving $(\cA,Z)$, where $\cA = \cP\left((0,1]\right)$ is the heart of a bounded t-structure, compatible with $Z$. 

Now let $(S,H)$ be a polarized surface. For every $\alpha\in \R$, we define a central charge on $\Db(S)$ as follows:
$$Z_{\alpha,\beta}(F)\coloneqq (-\ch_2(F)+\alpha H^2\rk(F))+i(H\ch_1(F)-\beta H^2\rk(F)).$$

\begin{Thm}[{\cite[Theorem 6.10]{Emolo-Benjamin:lecture-notes}}]\label{thm:tiltfamily}
For every $\alpha>\Phi_{S,H}(\beta)$, the pair $\sigma_{\alpha,\beta}=(\Coh^\beta(S),Z_{\alpha,\beta})$ is a geometric stability condition (with respect to $\Lambda_H$) on $S$. Moreover, the map
$$\Sigma:\{(\alpha,\beta)\in \R^2|\alpha>\Phi_{S,H}(\beta)\}\rightarrow \Stab_H(S): (\alpha,\beta)\mapsto \sigma_{\alpha,\beta}$$
is a continuous embedding.
\end{Thm}
\begin{Rem}\label{rem:familygeo}

The above statement is slightly stronger than \cite[Theorem 6.10]{Emolo-Benjamin:lecture-notes}, in the sense that we include stability conditions $\sigma_{\alpha,\beta}$ for some $\alpha\leq \frac{\beta^2}2$ when $\Phi_{S,H}(x)\neq \frac{x^2}2$. We briefly explain why the construction still works.

For every pair $(\alpha_0,\beta_0)$ satisfying $\alpha_0>\Phi_{S,H}(\beta_0)$, since $\Phi_{S,H}$ is upper semi-continuous and not greater than $\frac{x^2}2$, there exists sufficiently small $\delta> 0$  satisfying
$$\delta^{-1}(x-\beta_0)^2+\alpha_0-\delta\geq \Phi_{S,H}(x)$$ for every $x\in \R$. In other words, every $\mu_H$-semistable coherent sheaf $F$ satisfies the  Bogomolov type inequality:
$$\delta^{-1}\left(H\ch_1(F)-\beta_0 H^2\rk(F)\right)^2-H^2\rk(F)\left(\ch_2(F)-(\alpha_0-\delta)H^2\rk(F)\right)\geq 0.$$
By exactly the same argument as that for \cite[Theorem 6.10]{Emolo-Benjamin:lecture-notes} (or see
\cite[Section 2]{Dulip-Toda} for a more general set-up), the pair $\sigma_{\alpha_0,\beta_0}$ is a stability condition and the embedding $\Sigma$ is continuous at $(\alpha_0,\beta_0)$.

Note that skyscraper sheaves are all simple in $\Coh^{\beta_0}(S)$ by definition. They are all $\sigma_{\alpha_0,\beta_0}$-stable with phase $1$. Hence the stability condition $\sigma_{\alpha_0,\beta_0}$ is geometric.
\end{Rem}

\subsection{Geometric stability conditions on surfaces}
\begin{Prop}[cf. {\cite[Section 10]{Bridgeland:K3}}]
\label{prop:geomstabform}
Let $\sigma=(\cP,Z)$ be a geometric stability condition in $\Stab_H^{\mathrm{Geo}}(S)$. Then $\sigma=\sigma_{\alpha,\beta}g$ for some $\alpha>\Phi_{S,H}(\beta)$ and $g\in \glt$.
\end{Prop}
\begin{proof}
Applying an element of $\glt$ one can assume that $Z(\cO_p)=-1$ and $\cO_p\in \cP(1)$ for all $p\in X$. In particular, the central charge is of the form $$-\ch_2+aH\ch_1+bH^2\rk +i (cH\ch_1+dH^2\rk)$$ for some $a,b,c,d\in \R$.

By \cite[Lemma 10.1.c]{Bridgeland:K3}, the torsion sheaf $\cO_C(mH)\in \cP((0,1])$ for every curve $C$ on $S$ and $m\in \Z$. Therefore, the coefficient $c>0$. Applying an element of $\glt$ again, we may assume the central charge is of the form $Z_{\alpha,\beta}$ for some $\alpha,\beta\in \R$. By  \cite[Lemma 6.20]{Emolo-Benjamin:lecture-notes} (note that the argument holds for all $\alpha\in\R$), the heart $\cP((0,1])=\Coh^\beta(S)$. 

Now we only need to show that $\alpha>\Phi_{S,H}(\beta)$. By \cite[Proposition 5.27]{Emolo-Benjamin:lecture-notes}, there is an open neighbourhood $U$ of $\sigma$ in $\Stab_H(S)$ where all skyscraper sheaves are stable. By Theorem  \ref{thm:spaceasamfd}, there is an open neighborhood $W$ of $(\alpha,\beta)$ in $\R^2$ such that for every $(\alpha',\beta')\in W$, there exists a stability condition $\sigma'=(\cP',Z')\in U$ with $$\ker Z'=(1, \beta',\alpha')\cdot \R\subset \Lambda_H\otimes \R.$$

Suppose $\alpha\leq \Phi_{S,H}(\beta)$, by Definition \ref{def:lpfunction}, there exists a $\mu_H$-semistable sheaf $F$ and $(\alpha_0,\beta_0)\in W$ such that $$H\ch_1(F)=\beta_0 H^2\rk(F)\text{ and } \ch_2(F)> \alpha_0 H^2\rk(F).$$

Let $\sigma_0=(\cP_0,Z_0)$ be a stability condition in $U$ with $\ker Z_0=(1, \beta_0,\alpha_0)\cdot \R$. Applying an element of $\glt$ on $\sigma_0$, one can assume that $Z_0=Z_{\alpha_0,\beta_0}$. By  \cite[Lemma 6.20]{Emolo-Benjamin:lecture-notes}, $\cP_0((0,1])=\Coh^{\beta_0}(S)\ni F[1]$. This leads to a contradiction that $Z_{\alpha_0,\beta_0}(F[1])\in \R_{>0}$. Therefore, we must have $\alpha>\Phi_{S,H}(\beta)$.
\end{proof}

\begin{Cor}\label{cor:contractible}
Let $(S,H)$ be a smooth polarized surface such that its Albanese morphism is finite, then $\Stab_H(S)$ is connected and contractible. 
\end{Cor}
\begin{proof}
By Corollary \ref{cor:albfiniteimplygeom}, Proposition \ref{prop:geomstabform}, and \cite[Example 2.12(2)]{LahozRojaz}, $\Stab_H(S)$ is homeomorphic to a $\glt$-principal bundle over $\{(\alpha,\beta)\in\R^2|\alpha>\frac{\beta^2}2\}$.  It follows that $\Stab_H(S)$ is connected and contractible.
\end{proof}
\begin{Cor}\label{cor:irregsurf}
The space of stability conditions of an irregular surface with \nes rank one is connected and contractible.
\end{Cor}
\begin{proof}
The Albanese morphism of an irregular surface $S$ is non-constant. It does not contract any curve, since otherwise the \nes rank of $S$ is at least $2$. Therefore the Albanese morphism is finite. Let $H$ be an ample divisor, then $\Stab(S)=\Stab_H(S)$, which is connected and contractible by Corollary \ref{cor:contractible}.
\end{proof}

\section{Abelian threefold case}
\label{sec:AV3folds}
Let $(A,H)$ be a polarized abelian threefold. In this section, we show that the principal component of the stability manifold of $A$ constructed in \cite[Theorem 9.1]{BMS:stabCY3s} is the whole space $\Stab_H(A)$.

\subsection{Review: stability conditions on abelian threefolds}
We briefly recall the construction of stability conditions on abelian threefolds as that in \cite{Dulip-Antony:II,Dulip-Antony:I} and \cite{BMS:stabCY3s}.

For every $\beta\in \R$, recall from the previous section that we have the heart \[\Coh^\beta(A)\coloneqq \langle \cT_\beta,\cF_\beta[1]\rangle.\]

We will always consider the twisted Chern character $\ch^\beta(F)\coloneqq e^{-\beta H}\ch(F)$. More explicitly, we have
\begin{align*}
  &   \ch_1^\beta = \ch_1-\beta H\rk; \;\;\;\;\;\;\;\;\; \ch_2^\beta = \ch_2-\beta H\ch_1+\frac{\beta^2}2H^2\rk;\\
  & \ch_3^\beta = \ch_3-\beta H\ch_2+\frac{\beta^2}2H^2\ch_1-\frac{\beta^3}6H^3\rk.
\end{align*}
Note that for every object $F$ in $\Coh^\beta(A)$, we have $H^2\ch_1^\beta(F)\geq 0$.

For every $\alpha>0$, we consider the slope function $\mu_{\alpha,\beta}$ on $\Coh^\beta(X)$ defined as 
$$\mu_{\alpha,\beta}(F):=\begin{cases}\frac{H\ch^\beta_2(F)-\frac{\alpha^2}2\rk(F)}{H^2\ch^\beta_1(F)}, & \text{ if $H^2\ch^\beta_1(F)>0$;}\\
+\infty, & \text{ if $H^2\ch^\beta_1(F)=0$.}\end{cases}$$
A non-zero object $F\in \Coh^\beta(A)$ is called $\mu_{\alpha,\beta}$-(semi)stable if for every proper non-zero subobject $E\hookrightarrow F$ in $\Coh^\beta(A)$, one has 
$$\mu_{\alpha,\beta}(E)<(\leq)\mu_{\alpha,\beta}(F/E).$$

Consider the pair of subcategories:
\begin{align*}
    \cT_{\alpha,\beta} \coloneqq \{F\in \Coh^\beta(A)|\text{ any quotient object } F\twoheadrightarrow G \text{ in $\Coh^{\beta}(A)$ satisfies } \mu_{\alpha,\beta}(G)>0\}; \\ \cF_{\alpha,\beta} \coloneqq \{F\in \Coh^\beta(A)|\text{ any non-zero subobject } E\hookrightarrow F\text{ in $\Coh^{\beta}(A)$ satisfies } \mu_{\alpha,\beta}(E)\leq 0\}.
\end{align*}
This is a torsion pair on $\Coh^{\beta}(A)$. We denote the tilted heart as
$$\Coh^{\alpha,\beta}(A)\coloneqq \langle \cT_{\alpha,\beta} ,\cF_{\alpha,\beta} [1]\rangle.$$
For every $b\in\R$ and $a>\frac{\alpha^2}6+\frac{1}2\abs{b}\alpha$, we define the central charge as
$$Z^{a,b}_{\alpha,\beta}\coloneqq -\ch^\beta_3+bH\ch^\beta_2+aH^2\ch^\beta_1+i\left(H\ch^\beta_2-\frac{1}2\alpha^2H^3\rk\right).$$

We will  write $\mu_{\alpha,\beta}^{a,b}\coloneqq -\frac{\Re(Z^{a,b}_{\alpha,\beta})}{\Im(Z^{a,b}_{\alpha,\beta})}$ for the slope function. Now we may summarize the construction of stability conditions on $\Db(A)$ as follows.
\begin{Thm}[{\cite[Section 8 and Theorem 9.1]{BMS:stabCY3s}}]
\label{thm:bms}
Let $(A,H)$ be a polarized abelian threefold. Then for every $\alpha,\beta,a,b\in \R$ satisfying $\alpha>0$ and $a>\frac{1}6\alpha^2+\frac{1}2\abs{b}\alpha$, the pair $\sbb=(\Coh^{\beta}(A),Z^{a,b}_{\alpha,\beta})$ is a stability condition on $\Db(A)$ in $\Stab_H(A)$. Moreover, the map
\begin{align*}
    \Sigma: \{(a,b,\alpha,\beta)\in\R^4|\alpha>0,a>\frac{\alpha^2}6+\frac{1}2\abs{b}\alpha\} & \rightarrow \Stab_H(A) \\
    (a,b,\alpha,\beta) & \mapsto \sbb
\end{align*}
is a continuous embedding. 

Denote $\tilde{\mathfrak P}\coloneqq(\mathrm{im}(\Sigma))\cdot\glt$, then $\mathrm{im}(\Sigma)$ is a slice of the $\glt$-action on $\tilde{\mathfrak{P}}$. The space $\tilde{\mathfrak P}$ is a connected component in $\Stab_H(A)$.
\end{Thm}

Our goal is to show that  $\tilde{\mathfrak P}$ is actually the unique connected component of $\Stab_H(A)$.

\subsection{Semi-homogeneous vector bundles}
Let $(A,H)$ be a polarized abelian threefold. We will  make use of simple semi-homogeneous vector bundles with $\frac{\mathrm{det}}{\rk}=sH$ for $s\in\Q$. To simplify the notation, we will denote such a simple semi-homogeneous vector bundle as $E_s$.

Write $s=\frac{p}q$ for some coprime integers $p, q$ with $q\in\Z_{>0}$. The Chern characters of $E_s$ are
$$\ch(E_{s}) = \rk(E_{s}) e^{sH}=\rk(E_{s})(1,sH,\frac{s^2}2H^2,\frac{s^3}6H^3).$$

For every object $F\in \Db(A)$, it follows from the Hirzebruch--Riemann--Roch formula that
$$\chi(E_s,F)=\rk(E_s)\ch^s_3(F).$$

Every $E_s$ can be constructed as the push-forward of line bundles via an
isogeny $Y \to X$, see \cite[Theorem 5.8]{Mukai:semi-homogeneous}.

\begin{Lem}\label{lem:trivalepqhom}
Let $(A,H)$ be a polarized abelian threefold. If $s<t$, then $\Hom(E_s,E_t[i])\neq 0$ if and only if $i=0$.
\end{Lem}
\begin{proof}
By Corollary \ref{cor:EpqStable}, both $E_s$ and $E_t$ are $\sigma_{\alpha,\beta}^{a,b}$-stable  for $\beta=\frac{s+t}2$, $\alpha<\frac{t-s}2$,  $b=0$, and $a=\frac{\alpha^2}6+\epsilon$. In particular, $E_t,E_s[2]\in \Coh^{\beta,\alpha}(A)$, and 
$$\mu_{\alpha,\beta}^{\frac{\alpha^2}6+\epsilon,0}(E_s[2])=\frac{s-\beta}3<\frac{t-\beta}3=\mu_{\alpha,\beta}^{\frac{\alpha^2}6+\epsilon,0}(E_t)$$
when $\epsilon$ tends to $0$. Therefore, $\Hom(E_t,E_s[i])=0$ when $i\leq 2$. By Serre duality, $\Hom(E_s,E_t[i])=0$ unless $i=0$. 

By the Hirzebruch--Riemann--Roch formula, we have  $$\chi(E_s,E_t)=\rk(E_s)\rk(E_t)\frac{(t-s)^3}6>0.$$
Therefore, $\Hom(E_t,E_s)\neq 0$.
\end{proof}

Denote by $\cC$ the set of all (shifted) skyscraper sheaves $\cO_p[t]$ and all simple semi-homogeneous vector bundle $E_s[t]$'s. By Corollary \ref{cor:albfiniteimplygeom} and \ref{cor:EpqStable}, all elements in $\cC$ are stable with respect to every stability condition. The following proposition is the key input to our main result.

\begin{Prop}
\label{prop:ab3phasebound}
Let $\sigma=\sbb$ be a stability condition on $\Db(A)$ and $F\in \cP_\sigma(\theta)$. Then
\begin{align*}
    \inf_{E\in\cC}\left\{\theta-\phi_\sigma(E)|\Hom(E,F)\neq 0\right\}\leq 1; \\
     \inf_{E\in\cC}\left\{\phi_\sigma(E)-\theta|\Hom(F,E)\neq 0\right\}\leq 1.
\end{align*}
\end{Prop}
\begin{proof}
We may assume that $F$ is $\sigma$-stable and is not in $\cC$. Since the set $\cC$ is closed under homological shift, we may also assume $\theta\in(0,1]$. 

We first deal with the case when $\theta\in(0,1)$ and $H^{3-\bullet}\ch_\bullet(F)\not\in\R\cdot(1,t,\frac{t^2}2,\frac{t^3}6)$. Denote by $C\coloneqq\mu(F)=-\cot(\pi\theta)\in\R$ the slope of $F$. It follows that 
\begin{equation}\label{eqch3mu}
    \ch^\beta_3(F)-bH\ch^\beta_2(F)-aH^2\ch^\beta_1(F)-C\left(H\ch^\beta_2(F)-\frac{1}2\alpha^2H^3\rk(F)\right)=0.
\end{equation}
Consider the equation
\begin{equation}
    f(x)\coloneqq\frac{x^3}6-(C+b)\frac{x^2}2-ax+\frac{1}2\alpha^2C=0.
\end{equation}
By the assumption that $a>\frac{1}6\alpha^2+\frac{1}2\abs{b}\alpha$ and $\alpha>0$, we have $$f(\alpha)=\alpha\left(\frac{1}6\alpha^2-\frac{1}2b\alpha-a\right)<0<\alpha\left(a-\frac{1}6\alpha^2-\frac{1}2b\alpha\right)=f(-\alpha).$$
Note that $\lim_{x\rightarrow \pm\infty}f(x)=\pm\infty$, we have $$f(s_j)=0 \text{ for some }s_1<-\alpha<s_2<\alpha<s_3.$$

It is worth noticing that  $f(x)=0$ if and only if $\mu(e^{(\beta+x)H})=C$.

Now assume that  $\ch_3^{\beta+s_1}(F)>0$,  we will show that under this assumption, we have $$\inf_{E\in\cC}\left\{\phi_\sigma(E)-\theta|\Hom(F,E)\neq 0\right\}\leq1.$$ 

\begin{Rem}\textit{
When $\beta+s_1$ is rational, we may simply notice that} $$\chi(F,E_{\beta+s_1})=-\rk(E_{\beta+s_1})\ch_3^{\beta+s_1}(F)<0.$$
\textit{It then follows that  $\Hom(F,E_{\beta+s_1}[3])\neq 0$, which implies the conclusion.}

\textit{As $\beta+s_1$ can be an irrational number, we need to deform the stability condition  so that there exists $E_{\beta+s'_1}$ in $\cC$. Yet this deformation is not completely straightforward, as at certain point we need the assumption that $H^{3-\bullet}\ch_\bullet(F)\not\in\R\cdot(1,t,\frac{t^2}2,\frac{t^3}6)$. The only purpose of next few paragraphs is  to deal with this issue. For the convenience of the readers, it is harmless to skip these paragraphs and simply assume that $E_{\beta+s_i}\in\cC$ exists.}
\end{Rem}

By \cite[Proposition 5.27]{Emolo-Benjamin:lecture-notes}, there is an open neighbourhood $W$ of $(\alpha,\beta,a,b)$ such that $F$ is $\sigma_{\alpha',\beta'}^{a',b'}$-stable for every $(\alpha',\beta',a',b')\in W$. For every $\epsilon>0$, there exists $\delta_0>0$ such that for every $\abs{\delta},\abs{\delta'},\abs{\delta''}<\delta_0$ we have
\begin{itemize}
    \item $(\alpha+\delta'',\beta,a+\delta,b+\delta')\in W$;
    \item $\ch_3^{\beta+s_1+\delta}(F)>0$;
    \item  $\mathrm{dist}(\sbb,\sigma_{\alpha,\beta}^{a+\delta,b+\delta'})<\frac{\epsilon}2$.
\end{itemize}
It follows from  $f(s_1)=0$ that $\mu(e^{(\beta+s_1)H})=C$. Note that the function $B(x)$ satisfying $$\mu_{\alpha,\beta}^{a,B(x)}(e^{(\beta+x)H})=\mu_{\alpha,\beta}^{a,B(x)}(F)$$
is well-defined and continuous when $x\neq \pm \sqrt{\frac{2H\ch^\beta_2(F)}{H^3\rk(F)}}$. When $s_1\neq \pm\sqrt{\frac{2H\ch^\beta_2(F)}{H^3\rk(F)}}$, $B(s_1)=b$. There exists $\abs{\delta}<\delta_0$ such that $E_{\beta+s_1+\delta}\in\cC$ and $\abs{B(s_1+\delta)-b}<\delta_0$. 

If $s_1= \pm\sqrt{\frac{2H\ch^\beta_2(F)}{H^3\rk(F)}}$ but $s_1\neq \frac{H^2\ch_1^\beta(F)}{H^3\rk(F)}$, then the  function $A(x)$ satisfying  $$\mu_{\alpha,\beta}^{A(x),b}(e^{(\beta+x)H})=\mu_{\alpha,\beta}^{A(x),b}(F)$$
is well-defined and continuous at a small neighbourhood of $x=s_1$. There exists $\abs{\delta}<\delta_0$ such that $E_{\beta+s_1+\delta}\in\cC$ and $\abs{A(s_1+\delta)-a}<\delta_0$. 

If $s_1= \pm\sqrt{\frac{2H\ch^\beta_2(F)}{H^3\rk(F)}}= \frac{H^2\ch_1^\beta(F)}{H^3\rk(F)}$, since $H^{3-\bullet}\ch_\bullet(F)\not\in\R\cdot(1,t,\frac{t^2}2,\frac{t^3}6)$, we have $s_1^3\neq \frac{6\ch_3^\beta(F)}{H^3\rk(F)}$. It follows that the positive function  $P(x)$ satisfying  $$\mu_{P(x),\beta}^{a,b}(e^{(\beta+x)H})=\mu_{P(x),\beta}^{a,b}(F)$$
is well-defined  and continuous at a small neighbourhood of $x=s_1$ (as the denominator of $\frac{1}2(P(x))^2$ is $\frac{s^3_1}6H^3\rk(F)-\ch_3^\beta(F)$ at $x=s_1$). There exists $\abs{\delta}<\delta_0$ such that $E_{\beta+s_1+\delta}\in\cC$ and $\abs{P(s_1+\delta)-\alpha}<\delta_0$.\\

In any case, we may assume there exists $E_{\beta+s'_1}\in \cC$ and $(\alpha',\beta,a',b')$ such that
\begin{itemize}
    \item $s'_1<-\alpha'$, $E_{\beta+s'_1}[2]\in\cP_\sigma((0,1])$ and $\ch_3^{\beta+s'_1}(F)>0$;
    \item $F$ is $\sigma_{\alpha',\beta}^{a,b}$-stable and $\mu_{\alpha',\beta}^{a',b'}(F)=\mu_{\alpha',\beta}^{a',b'}(E_{\beta+s'_1})$;
    \item $\abs{\phi_\sigma(E_{\beta+s'_1}[2])-\phi_\sigma(F)}<\epsilon$.
\end{itemize}

Note that $\phi_{\alpha',\beta}^{a',b'}(F)=\phi_{\alpha',\beta}^{a',b'}(E_{\beta+s'_1}[2])$ and $F\neq E_{\beta+s'_1}[2]$, we have
$$\Hom(E_{\beta+s'_1},F[t])=(\Hom(F,E_{\beta+s'_1}[3-t]))^*=0$$
for all $t\neq 0,-1$. Therefore,
\begin{align*}
    & \Hom(F,E_{\beta+s'_1}[4])-\Hom(F,E_{\beta+s'_1}[3]) \\ = & \chi(F,E_{\beta+s'_1})=-\rk(E_{\beta+s'_1})\ch^{\beta+s'_1}(F)<0.
\end{align*}
It follows that $\Hom(F,E_{\beta+s'_1}[3])\neq 0$. In particular, 
$$\inf_{E\in\cC}\left\{\phi_\sigma(E)-\theta|\Hom(F,E)\neq 0\right\}\leq \phi_\sigma(E_{\beta+s'_1}[3])-\theta< 1+\epsilon.$$
As $\epsilon$ tends to $0$,  we have $\inf_{E\in\cC}\left\{\phi_\sigma(E)-\theta|\Hom(F,E)\neq 0\right\}\leq 1$.\\

By the same argument, if any value of $\{(-1)^{i+1}\ch_3^{\beta+s_i}(F)\}_{i=1,2,3}$ is positive, then $$\inf_{E\in\cC}\left\{\phi_\sigma(E)-\theta|\Hom(F,E)\neq 0\right\}\leq 1.$$

If any value of $\{(-1)^{i+1}\ch_3^{\beta+s_i}(F)\}_{i=1,2,3}$ is negative, then $$\inf_{E\in\cC}\left\{\theta-\phi_\sigma(E)|\Hom(E,F)\neq 0\right\}\leq 1.$$ 

To conclude the argument for the case that $H^{3-\bullet}\ch_\bullet(F)\not\in\R\cdot(1,t,\frac{t^2}2,\frac{t^3}6)$ and $\theta\neq 1$, we need the following property for the values of $\ch_3^{\beta+s_i}(F)$.
\begin{Lem}
\label{lem:ch3sF}
The values $\{(-1)^j\ch^{\beta+s_i}_3(F)\}_{i=1,2,3}$ cannot be all positive (or negative) at the same. Moreover, if all of them are non-negative (or non-positive), then they must be all $0$ and  
$$(H^3\rk(F),H^2\ch_1^\beta(F),H\ch_2^\beta(F),\ch_3^\beta(F))=\lambda\cdot(1, b-C,-a,-\frac{C}2\alpha^2)$$ for some $\lambda\in \R^\times$.
\end{Lem}
The proof for the lemma is  elementary, we postpone it after the proof of this proposition. \\

By Lemma \ref{lem:ch3sF}, the only outstanding case is when $\ch^{\beta+s_i}_3(F)$ are all zero. In this case, for every $\epsilon>0$, we may deform $a$ to an $a'\neq a$ such that
\begin{itemize}
    \item $\mathrm{dist}(\sigma,\sigma_{\alpha,\beta}^{a',b})<\epsilon$;
    \item $F$ is $\sigma_{\alpha,\beta}^{a',b}$-stable and is not in $\cP_{\sigma_{\alpha,\beta}^{a',b}}(1)$.
\end{itemize}
The characters of $F$ cannot satisfy the condition $$(H^3\rk(F),H\ch_2^\beta(F))=\lambda\cdot(1, -a')$$  as that in Lemma \ref{lem:ch3sF} for any $\lambda\in\R^\times$. Therefore, 
\begin{align*}
& \inf_{E\in\cC}\left\{\phi_{\sigma}(E)-\phi_{\sigma}(F)|\Hom(F,E)\neq 0\right\}\\
< & 
\inf_{E\in\cC}\left\{\phi_{\sigma_{\alpha,\beta}^{a',b}}(E)-\phi_{\sigma_{\alpha,\beta}^{a',b}}(F)+2\epsilon\;\middle| \;\Hom(F,E)\neq 0\right\}\leq 1+2\epsilon.
\end{align*}

As $\epsilon$ tends to $0$, we get the inequality  for $\sigma$. The statement holds for all $F\in\cP_\sigma(\theta)$ with $\theta\not\in\Z$ and $H^{3-\bullet}\ch_\bullet(F)\neq\R\cdot(1,t,\frac{t^2}2,\frac{t^3}6)$.\\

In the case that $\theta=1$, if $F$ is a skyscraper sheaf, then the statement hold automatically as $F\in\cC$. Otherwise, $\ch(F)\neq (0,0,0,a)$, we can deform $\sigma$ in any small open neighbourhood so that $F$ is still stable but not with phase $1$. The inequalities in the statement hold.\\

In the case that $H^{3-\bullet}\ch_\bullet(F)=\R\cdot(1,t,\frac{t^2}2,\frac{t^3}6)$ for some $t\in \R$, we have $Q_K^\beta(H^{3-\bullet}\ch_\bullet(F))=0$ for all quadratic form $Q_K^\beta=K\Delta_H+\nabla^\beta_H$. By  \cite[Proposition A.8]{BMS:stabCY3s}, they are stable with respect to all stability conditions as that in Theorem \ref{thm:bms}. In particular, $F[-i]=F'\in \Coh(A)$ where $i=0$ when $t>\beta+\alpha$; $i=1$ when $\beta-\alpha<t\leq\beta+\alpha$; and  $i=2$ when $t\leq\beta-\alpha$. By the same argument as that in Lemma \ref{lem:trivalepqhom}, $\Hom(E_s,F')\neq 0$ when $s<t$; and $\Hom(F',E_s)\neq 0$ when $s>t$. It is clear that both infimums for $F$ in the statement are $0$.
\end{proof}

\begin{proof}[{Proof for Lemma \ref{lem:ch3sF}}]
As $s_i$'s are the solutions to $f(x)=0$, we have
\begin{align}
    s_1+s_2+s_3&=3b-3C ;\\ 
     s_1s_2+s_2s_3+s_3s_1&=-6a ;\\
     s_1s_2s_3&=3\alpha^2C.
\end{align}

Substitute \eqref{eqch3mu} into $\ch^{\beta+s_i}_3(F)$'s, we may replace the term $\ch^\beta_3(F)$ by other terms:
\begin{equation}
    \ch^{\beta+s_i}_3(F)=(b-C-s_i)H\ch^\beta_2(F)+(a+\frac{s_i^2}2)H^2\ch^\beta_1(F) +(\frac{1}2C\alpha^2-\frac{s_i^3}6)H^3\rk(F).
\end{equation}
To simplify the notations, we denote
\begin{equation}\label{eqlines}
    L_i(x,y,z)\coloneqq(b-C-s_i)x+(a+\frac{s_i^2}2)y +(\frac{1}2C\alpha^2-\frac{s_i^3}6)z.
\end{equation}
Substitute $(x,y,z)=(\frac{s^2_i}2,s_i,1)$ into $L_j$'s, we have
\begin{align*}
    L_j(\frac{s^2_i}2,s_i,1) & = (b-C)\frac{s_i^2}2-\frac{s_js^2_i}2+as_i+\frac{s^2_js_i}2+\frac{\alpha^2C}2-\frac{s^3_j}6 \\ 
    & =\frac{s^3_i}6-\frac{s_js^2_i}2+\frac{s^2_js_i}2-\frac{s^3_j}6\\
    & = \frac{1}6 (s_i-s_j)^3.
\end{align*}
Substitute $(x,y,z)=(-a,b-C,1)$ into   $L_i$'s, we have
\begin{align*}
    L_i(-a,b-C,1) & = as_i+\frac{s_i^2}2(b-C)+\frac{1}2C\alpha^2-\frac{s^3_i}6 \\
    & = -\frac{1}6\left((s_1s_2+s_2s_3+s_3s_1)s_i-s_i^2(s_1+s_2+s_3)-s_1s_2s_3+s_i^3\right)\\
    & =0.
\end{align*}
Therefore, $L_i$'s are linear dependant, and we have
$$(s_3-s_1)^3L_2=(s_3-s_2)^3L_1+(s_2-s_1)^3L_3.$$
Since $s_1<s_2<s_3$, the values $\{(-1)^iL_i(x,y,z)\}$ are all non-negative (or non-positive) if and only if $(x,y,z)\in\R\cdot(-a,b-C,1)$.
\end{proof}
\subsection{Stability manifold of abelian threefolds}
Given stability conditions $\sigma_1$ and $\sigma_2$ on $\Db(X)$, we may consider the following generalised metric function on them as that defined in \cite[Section 6]{Bridgeland:Stab}:
$$d(\sigma_1,\sigma_2)\coloneqq\sup_{0\neq E\in \Db(X)}\{\abs{\phi_{\sigma_1}^-(E)-\phi_{\sigma_2}^-(E)},\abs{\phi_{\sigma_1}^+(E)-\phi_{\sigma_2}^+(E)}\}\in [0,+\infty].$$

\begin{Lem}
\label{lem:distsym}
Let $\sigma$ and $\tau$ be two stability conditions, $\delta\in[0,1]$, and $\cD$ be a set of objects such that
\begin{enumerate}
    \item every object $E$ in $\cD$ is both $\sigma$-stable and $\tau$-stable with $\phi_\sigma(E)=\phi_\tau(E)$;
 \item for every object $F\in\cP_\sigma(\theta)$, we have 
\begin{align*}
    \inf_{E\in\cD}\left\{\theta-\phi_\sigma(E)|\Hom(E,F)\neq 0\right\}\leq \delta; \\
     \inf_{E\in\cD}\left\{\phi_\sigma(E)-\theta|\Hom(F,E)\neq 0\right\}\leq \delta.
\end{align*}
\end{enumerate}
Then $d(\sigma,\tau)\leq \delta$.
\end{Lem}
\begin{proof}
For every $\tau$-stable object $F$, we have the distinguished triangle
\begin{equation}\label{eq1}
    E\rightarrow F\rightarrow G\xrightarrow{+},  \end{equation}
    
 where $E=\mathrm{HN}^{\phi^+_\sigma(F)}_{\sigma}(F)$  and $G= \mathrm{HN}^{<\phi^+_\sigma(F)}_{\sigma}(F)$. 
 
 Suppose $\phi^+_\sigma(F)> \phi_\tau(F)+\delta$, then it follows that
 $$\phi_\sigma(E)=\phi_\sigma^+(F)>\max\{\phi_\tau(F)+\delta,\phi^+_\sigma(G)\}.$$
 
 Since $E$ is $\sigma$-semistable, by the assumption, there exists $A\in\cD$ such that
 $$\Hom(A,E)\neq 0 \text{ and } \phi_\tau(A)=\phi_\sigma(A)>\max\{\phi_\tau(F),\phi^+_\sigma(G)-\delta\}.$$
 Apply $\Hom(A,-)$ to \eqref{eq1}, we have an exact sequence 
 $$\dots\rightarrow \Hom(A,G[-1])\rightarrow\Hom(A,E)\rightarrow\Hom(A,F)\rightarrow\dots .$$
 Since both $A$ and $F$ are $\tau$-stable and $\phi_\tau(A)>\phi_\tau(F)$, $\Hom(A,F)=0$. Since $A$ is $\sigma$-stable and $\phi^+_\sigma(G[-1])=\phi^+_\sigma(G)-1\leq \phi^+_\sigma(G)-\delta<\phi_\sigma(A)$, we have $\Hom(A,G[-1])=0$.  This  contradicts to the fact that $\Hom(A,E)\neq 0$ and the exactness of the sequence. Therefore, we must have $\phi^+_\sigma(F)\leq\phi_\tau(F)+\delta$ for every $\tau$-stable object $F$. 
 
 For the same reason, $\phi^-_\sigma(F)\geq\phi_\tau(F)-\delta$ for every $\tau$-stable object $F$. By \cite[Lemma 6.1]{Bridgeland:Stab}, we have $d(\sigma,\tau)\leq \delta$.
\end{proof}

\begin{Lem}
\label{lem:samecentralcharge}
Let $\sigma$ and $\tau$ be two stability conditions with the same central charge. If  $d(\sigma,\tau)\leq 1$, then $\sigma=\tau$.
\end{Lem}
\begin{proof}
For every $\sigma$-stable object $F$, we first show that $\phi^+_\tau(F)\geq \phi_\sigma(F)$. 

Suppose $\phi^+_\tau(F)<\phi_\sigma(F)$, then since $d(\sigma,\tau)\leq 1$, the object $F\in \cP_\tau\big([\phi_\sigma(F)-1,\phi_\sigma(F))\big)$. Therefore,
$$Z_\tau(F)= me^{\pi i\theta}\text{ for some }m>0\text{ and } \theta\in [\phi_\sigma(F)-1,\phi_\sigma(F)),$$
which cannot equal  $Z_\sigma(F)= m'e^{\pi i\phi_\sigma(F)}$ for any $m'>0$. It follows that $\phi^+_\tau(F)\geq \phi_\sigma(F)$. Due to the same argument, for every $\tau$-stable object $E$, $\phi_\sigma^+(E)\geq \phi_\tau(E)$. \\

For every $\sigma$-stable object $F$, we then show that $\phi^+_\tau(F)= \phi_\sigma(F)$. Note that we have the distinguished triangle
\begin{equation}\label{eq2}
    E\rightarrow F\rightarrow G\xrightarrow{+},  \end{equation}
    
 where $E=\mathrm{HN}^{\phi^+_\tau(F)}_{\tau}(F)$  and $G= \mathrm{HN}^{<\phi^+_\tau(F)}_{\tau}(F)$. Suppose $\phi^+_\tau(F)>\phi_\sigma(F)$, then 
 $$\phi^+_\sigma(E)\geq \phi_\tau(E)=\phi^+_\tau(F)>\phi_\sigma(F).$$
 In particular, we have $\Hom(\mathrm{HN}^{\phi^+_\sigma(E)}_{\sigma}(E),F)=0$. 
 
 Since $d(\sigma,\tau)\leq 1$, we have $$\phi^+_\sigma(G[-1])\leq \phi^+_\tau(G)<\phi^+_\tau(F)=\phi_\tau(E)\leq \phi^+_\sigma(E).$$ Therefore, we have $\Hom(\mathrm{HN}^{\phi^+_\sigma(E)}_{\sigma}(E),G[-1])=0$. Apply $\Hom(\mathrm{HN}^{\phi^+_\sigma(E)}_{\sigma}(E),-)$ to \eqref{eq2}, we get the contradiction. Hence, we must have $\phi^+_\tau(F)=\phi_\sigma(F)$.
 
 By the same argument, $\phi^-_\tau(F)=\phi_\sigma(F)$ for every $\sigma$-stable object $F$. By \cite[Lemma 6.1]{Bridgeland:Stab}, $d(\sigma,\tau)=0$. As $\sigma$ and $\tau$ also have the same central charge, they are the same stability condition.
\end{proof}

Now we are ready to prove our main result for the stability conditions on abelian threefolds.

\begin{Thm}\label{thm:ab3space}
Let $(A,H)$ be a polarized abelian threefold and $\sigma=(\cP,Z)$ be a stability condition in $\Stab_H(A)$. Then $\sigma=\sigma^{a,b}_{\alpha,\beta}g$ for some $\alpha>0$, $a>\frac{1}6\alpha^2+\frac{1}2\abs{b}\alpha$, and $g\in \glt$.
\end{Thm}
\begin{proof}
Applying an element of $\glt$ one can assume that $Z(\cO_p)=-1$ and $\cO_p\in \cP(1)$ for all $p\in A$. In particular, the central charge is of the form $$Z=-\ch_3+l_1 H\ch_2+l_2H^2\ch_1+l_3H^3\rk +i (m_1H\ch_2+m_2H^2\ch_1+m_3H^3\rk)$$ for some $l_i,m_i\in \R$.\\

We first show that $m_1>0$. Suppose $m_1\leq 0$, then we may deform the stability condition so that $\cO_p\in\cP(1)$ and $m_1<0$. Note that 
$$\lim_{t\rightarrow\pm\infty}\Im(Z(e^{tH}))=-\infty,$$ it follows from Corollary \ref{cor:EpqStable}, Lemma  \ref{lem:geohearts} that all $E_t\in\cC$ are in the heart $\cP((-1,0])$. Note that $$\lim_{t\rightarrow +\infty}-\frac{\Re (Z(e^{tH}))}{\Im (Z(e^{tH}))}=-\infty \text{ and }\lim_{t\rightarrow -\infty}-\frac{\Re (Z(e^{tH}))}{\Im (Z(e^{tH}))}=+\infty, $$
we have $\lim_{t\rightarrow +\infty}(\phi_{\sigma}(E_t))=-1$ and $\lim_{t\rightarrow -\infty}(\phi_{\sigma}(E_t))=0$. This contradicts with Corollary \ref{cor:EpqStable} and Lemma \ref{lem:trivalepqhom}. Therefore, we must have $m_1>0$.\\

Applying an element of $\glt$ we may further assume that the central charge is of the form $$Z=-\ch^\beta_3+b H\ch^\beta_2+aH^2\ch^\beta_1+cH^3\rk +i (H\ch^\beta_2+dH^3\rk)$$ for some $a,b,c,d,\beta\in \R$.

We then show that $d$ must be negative. Suppose $d\geq 0$, we may deform the stability condition so that $d>0$. Note that $\Im(Z(e^{tH}))>0$ for all $t\in\Q$, if follows from Corollary \ref{cor:EpqStable}, Lemma \ref{lem:geohearts} and \ref{lem:trivalepqhom} that there exists $\beta'\in \R\cup\{\pm\infty\}$ such that $E_t\in \cP((-2,-1])$ for all $t<\beta'$ and $E_t\in \cP((0,1])$ for all $t>\beta'$. 

We show that $\beta'$ must be $\beta$ in this case. Suppose $\beta'<\beta$, then there exists $\frac{p}q$ such that $\beta'<\frac{p}q<\frac{p+1}q<\beta$, and both $p$ and $p+1$ are coprime with $q$. Let $E_{\frac{p}q}, E_{\frac{p+1}q}\in\cC$. Since both of them are the push-forward of line bundle from the same isogeny map, there exists an injective map $f:E_{\frac{p}q}\rightarrow E_{\frac{p+1}q}$. Since the sheaf $F\coloneqq\mathrm{coker}(f)$ is the extension of $E_{\frac{p+1}q}$ and $E_{\frac{p}q}[1]$, it must be in $\cP((0,2])$. By Lemma \ref{lem:geohearts} part \emph{(b)}, the coherent sheaf $F$ is in $\cP((0,1])$. We get the contradiction by computing $$\Im(Z(F))=\Im(Z(E_{\frac{p+1}q}))-\Im(Z(E_{\frac{p}q}))=\frac{1}{2q}\left(\frac{2p+1}q-2\beta\right)H^3\rk(E_\frac{p}{q})<0.$$

Suppose $\beta'>\beta$, then there exists $\frac{p}q$ such that $\beta<\frac{p}q<\frac{p+1}q<\beta'$, and both $p$ and $p+1$ are coprime with $q$.  Since the sheaf $F$ is the extension of $E_{\frac{p+1}q}$ and $E_{\frac{p}q}[1]$, it must be in $\cP((-2,0])$. As $F$ is a torsion sheaf, it is in $\cP((-1,0])$ by Lemma \ref{lem:geohearts} part \emph{(b)}.  We get the contradiction by computing $$\Im(Z(F))=\Im(Z(E_{\frac{p+1}q}))-\Im(Z(E_{\frac{p}q}))=\frac{1}{2q}\left(\frac{2p+1}q-2\beta\right)H^3\rk(E_\frac{p}{q})>0.$$

Therefore, $\beta'=\beta$. In particular, for any $s<\beta<t$, we have $\phi_\sigma(E_s)\leq -1<0<\phi_\sigma(E_t)$. We may deform the stability condition in a sufficiently small neighbourhood to another stability condition $\sigma'$ such that $\cO_p\in\cP(1)$ and $\Im Z'=H\ch^{\beta_0}_2+d'H^3\rk$ for some $\beta_0> \beta$ and $d'>0$. By the same argument, we have $\phi_{\sigma'}(E_t)\leq -1<0<\phi_{\sigma}(E_t)$ for any rational number $\beta <t<\beta_0$. Hence, by \cite[Lemma 6.1]{Bridgeland:Stab}, we have $\mathrm{dist}(\sigma,\sigma')\geq 1$, and this leads to the contradiction.\\

Applying an element of $\glt$ we may further assume that the central charge is of the form $$Z=-\ch^\beta_3+b H\ch^\beta_2+aH^2\ch^\beta_1  +i (H\ch^\beta_2-\frac{\alpha^2}2H^3\rk)$$ for some $a,b,\beta\in \R$ and $\alpha>0$.

We now show that $a> \frac{\alpha^2}6+\frac{1}2\abs{b}\alpha$. Suppose $a\leq \frac{\alpha^2}6+\frac{1}2\abs{b}\alpha$, by deforming the stability condition, we may assume  $\alpha,\beta\in \Q$, and $a<\frac{\alpha^2}6+\frac{1}2\abs{b}\alpha$. 
Note that $\Im(Z(E_t))\geq 0$ when and only when $\abs{t-\beta}\geq \alpha$. By Corollary \ref{cor:EpqStable}, Lemma  \ref{lem:geohearts} and \ref{lem:trivalepqhom}, we have $E_t\in \cP((0,1])$ when $t>\beta+\alpha$; $E_t\in \cP((-1,0])$ when $\beta-\alpha<t\leq \beta+\alpha$; and $E_t\in \cP((-2,-1])$ when $t\leq \beta-\alpha$. Note that $\Im(Z(E_{\beta\pm\alpha}))=0$ and 
\begin{align*}
    \Re(Z(E_{\beta+\alpha}[1]))& =  -\rk(E_{\beta+\alpha})H^3\alpha (a+\frac{1}2b\alpha-\frac{\alpha^2}6);\\
    \Re(Z(E_{\beta-\alpha}[2]))& =  -\rk(E_{\beta-\alpha})H^3\alpha (a-\frac{1}2b\alpha-\frac{\alpha^2}6).
\end{align*}
At least one of them is positive, which leads to the contradiction.\\

Finally, applying an element of $\glt$, we may assume that $\cO_p\in\cP(1)$, and $\sigma$ has the same central charge as that of $\sbb$ for some $\alpha>0$ and $a>\frac{\alpha^2}6+\frac{1}2\abs{b}\alpha$. Moreover, by Corollary \ref{cor:EpqStable}, Lemma \ref{lem:trivalepqhom} and \ref{lem:geohearts}, for every object $E\in\cC$, we have $\phi_\sigma(E)=\phi_{\sbb}(E)$. 

By Proposition \ref{prop:ab3phasebound} and Lemma \ref{lem:distsym}, we have $d(\sigma,\sbb)\leq 1$. By Lemma \ref{lem:samecentralcharge}, we have $\sigma=\sbb$.
\end{proof}
\begin{Cor}\label{cor:ab3space}
Let $(A,H)$ be a polarized abelian threefold, then $\Stab_H(A)=\tilde {\mathfrak P}$.
\end{Cor}

\subsection{Further questions}

\begin{Ques}
Let $(X,H)$ be a smooth projective variety whose Albanese morphism is finite, then is $\Stab_H(X)$ also contractible?
\label{ques:contracthigh}
\end{Ques}
This is already non-trivial for the threefold case. Unlike Theorem \ref{thm:ab3space} of the abelian threefold case, there are examples of different geometric stability conditions on the projective space with the same central charge. However, we do not expect such examples exist when the Albanese morphism of the variety is finite. In particular, we expect \cite[Conjecture 1.7]{BMS:stabCY3s} to hold in an even stronger sense that $\tilde{\mathfrak P}_n$ is the whole space $\Stab_H(X)$ when the Albanese morphism of $X$ is finite.

\begin{Ques}
\label{ques:geoiffalb}
Let $X$ be a smooth projective variety whose Albanese morphism is not finite. Then does there always exist non-geometric stability conditions?
\end{Ques}
In other words, we expect the finite Albanese morphism condition is necessary and sufficient for all stability conditions being geometric. When $X$ is of dimension one, the answer is affirmative by \cite[Theorem 2.7]{Macri:curves}. When $X$ is of dimension greater than or equal to three, the question is far beyond reach, as even the existence of stability conditions are only  known in few cases, see \cite{BMS:stabCY3s, Zhiji-stabneftan, quintic, yucheng:stabprodvar} for more details.

This makes the surface case the most interesting one, which is nevertheless highly non-trivial.  By gluing stability conditions with respect to the Orlov semiorthogonal decomposition for blow-ups, there are always non-geometric stability conditions on  non-minimal surfaces. One may therefore always assume that the surface is minimal. Among all such surfaces, the most interesting case is when the Albanese morphism of $X$ is trivial, in other words, $H^1(X,\cO_X)=0$. We have the following conjecture for the Le Poitier function of them.\\

\noindent\textbf{Conjecture.} Let $(S,H)$ be a smooth polarized surface with zero irregularity, then the Le Potier function $\Phi_{S,H}$ is not continuous at $0$.\\

Admitting this conjecture, we expect that there always exist stability conditions as that discussed in \cite[Theorem 12.1]{Bridgeland:K3} when $S$ has zero irregularity. We will investigate this direction in a  future project.

\bibliography{all}                      
\bibliographystyle{halpha}  
\end{document}